\documentclass[12pt,a4paper]{amsart}
\usepackage{amssymb, amstext, amscd, amsmath, color}

\usepackage{url}

\usepackage[dvips]{graphicx}

 \newtheorem{thm}{Theorem}[section]
 \newtheorem{cor}[thm]{Corollary}

 \theoremstyle{definition}
 \newtheorem{defn}[thm]{Definition}
 \theoremstyle{remark}

 \numberwithin{equation}{section}

\newcommand{\bC}{{\mathbb{C}}}

\newcommand{\bR}{{\mathbb{R}}}
\newcommand{\bT}{{\mathbb{T}}}

\newcommand{\bZ}{{\mathbb{Z}}}

\newcommand{\ol}{\overline}

\newcommand{\sevone}{{\mathcal{C}_{\text{sev}_1}}}
\newcommand{\sevtwo}{{\mathcal{C}_{\text{sev}_2}}}
\newcommand{\sevthree}{{\mathcal{C}_{\text{sev}_3}}}
\newcommand{\Cknet}{{\mathcal{C}_{\text{knet}}}}
\newcommand{\Ckag}{{\mathcal{C}_{\text{kag}}}}

\newcommand{\Csod}{\mathcal{C}_{\text{SOD}}}

 \newcommand{\C}{{\mathcal{C}}}
 \newcommand{\D}{{\mathcal{D}}}

\renewcommand{\H}{{\mathcal{H}}}

 \newcommand{\K}{{\mathcal{K}}}

 \newcommand{\M}{{\mathcal{M}}}

 \newcommand{\T}{{\mathcal{T}}}
 \newcommand{\U}{{\mathcal{U}}}

\newcommand{\rank}{\operatorname{rank}}

\begin{document}
%
%
%
%
%
%
%
%
%
\title[Crystal frameworks and rigidity operators]
 {Crystal frameworks, matrix-valued functions and rigidity operators}\footnote{International Workshop on Operator Theory and Applications, Seville, July 2011}
\author[S.C. Power]{S.C. Power}

\address{Department of Mathematics and Statistics\\ Lancaster University\\
Lancaster LA1 4YF \\U.K. }

\email{s.power@lancaster.ac.uk}

\subjclass{Primary 52C75; Secondary 46T20}

\keywords{Crystal framework, rigidity operator, matrix function, rigid unit mode}

\date{July 2011}

\begin{abstract}
An introduction and survey is given of some recent work on the infinitesimal dynamics
of \textit{crystal frameworks}, that is, of  translationally periodic discrete bond-node structures in $\mathbb{R}^d$, for $ d=2,3,\dots $.
We discuss the rigidity matrix, a fundamental object from finite bar-joint framework theory,
rigidity operators, matrix-function representations and low energy phonons.
These phonons in material crystals, such as quartz and zeolites, are known as rigid unit modes, or RUMs,
and are associated with the relative motions of rigid units, such as ~SiO$_4$ tetrahedra in the tetrahedral polyhedral bond-node model for quartz.
We also introduce
semi-infinite crystal frameworks, bi-crystal frameworks and  associated multi-variable Toeplitz operators.
\end{abstract}

\maketitle

\section{Introduction}
A survey is given of some recent work on the infinitesimal dynamics of \textit{crystal frameworks}, by which we mean translationally periodic discrete bar-joint frameworks in $\bR^d$.
This includes a discussion of
rigidity operators,  matrix symbol function representations and the connections with models for low energy phonon modes in various material crystals. These modes are also known as rigid unit modes, or RUMs, reflecting their origin in the relative motion of
rigid units in the crystalline structure.
I also introduce briefly the contexts of semi-infinite {crystal frameworks} and bicrystal frameworks
and indicate how their rigidity operators involve multivariable Toeplitz operators whose symbol functions are matrices over multi-variable trigonometric polynomials on the $d$-torus.

The topic of infinite bar-joint frameworks, whether periodic or not, can be pursued as a purely mathematical endeavour and many aspects of deformability and rigidity remain to be understood.
The main perspectives below and related issues are  developed in Owen and Power \cite{owe-pow-hon}, \cite{owe-pow-fsr}, \cite{owe-pow-crystal} and Power \cite{pow-matrix}, \cite{pow-aff}.

Translationally periodic bond/node bar-joint frameworks or networks are ubiquitous
in mathematics (periodic tilings for example), solid state physics (crystal lattices, graphene), solid state chemistry (zeolites) and material science (microporous metal organic frameworks). So there is no lack of interesting examples.
I shall illustrate a number of concepts with three examples derived from
tilings seen in Seville at the Alcazaar and the Cathedral.

\section{Models for material crystals and low energy phonons.}
We begin by outlining one particualr motivation from material science.
A crystal framework $\C$ in $\bR^3$ can serve as a mathematical model for the essential geometry of the disposition of atoms and bonds in a material crystal $\M$.
In the model of interest to us
the vertices correspond to certain atoms while the edges  correspond
in some way to strong bonds. Also the identification of strongly bonded  "units" in $\M$ imply a polyhedral net structure and it is this that gives the relevant abstract framework $\C$. A fundamental  example of this kind is quartz, SiO$_2$,
in which each silicon atom lies at the centre of a strongly bonded  SiO$_4$ unit, which in turn may be modeled as a tetrahedron with
an  oxygen atom at each vertex.
In this way the material crystal quartz provides a mathematical crystal framework
of pairwise connected tetrahedra with a particular connectedness
and geometry.

Material scientists are interested in the manifestation and explanation of various forms of low energy motion and oscillation in materials.
Of particular interest are the
rigid unit modes  in aluminosilicate crystals and zeolites,
where quite complicated tetrahedral net models are relevant.
These  low energy (long wavelength) phonon modes
are observed in neutron
scattering experiments and have been shown to correlate closely with the modes observed in computational simulations.
There is now a considerable body of literature tabulating the (reduced) wave vectors of RUMs  of various crystals as subsets of the unit cube (Brillouin zone)
and it has become evident that the primary determinant  is the geometric structure of the abstract frameworks $\C$. See, for example, Dove et al \cite{dov-exotic}, Hammond et al \cite{ham-dov-zeo1997}, \cite{hamdov-zeo1998},
Giddy et al \cite{gid-et-al} and Swainson and Dove \cite{swa-dov}.
Particularly intriguing is the simulation study in Dove et al \cite{dov-exotic} which gives
a range of pictures of the RUM spectrum and multiplicites for various idealized crystal types.

In the experiments and in the simulations the background mathematical model is classical lattice
dynamics and rigid unit modes are observed where the phonon dispersion curves indicate vanishing energy.
However, one can also identify such limiting cases through a direct linear approach as we outline below and from this it follows that these sets (at least in simulations) may be viewed as real algebraic varieties.
See Theorem \ref{t:phaserigiditymatrix} below, \cite{owe-pow-crystal}, \cite{pow-matrix} and Wegner \cite{weg}.
It is convenient to define the RUM dimension to be the dimension of this algebraic variety.
(See Section 5.)
In 3D it takes the values $0,1,2,3$.

\section{An illustrative example}
The following simple example will serve well to illustrate
the notation scheme for general crystal frameworks in $d$ dimensions
that we adopt. The example is also of
interest in its own right, as we see later.

Figure 1 indicates a translationally periodic  bar-joint framework $\C=(G,p)$
determined by a sequence
$p=(p_k)$ in $\bR^2$. The framework edges $[p_i,p_j]$, associated with the edges
of the underlying graph $G$, are viewed as inextensible bars connected at the framework vertices
$p_k$ but otherwise unconstrained.

\begin{center}
\begin{figure}[h]
\centering
\includegraphics[width=7cm]{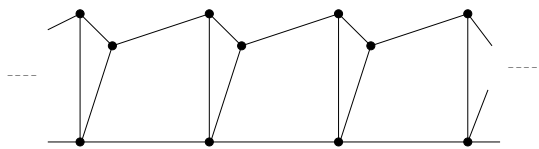}
\caption{An infinite bar-joint framework.}
\end{figure}
\end{center}

Let the scaling be such that
$$
p_1=p_{1,0}=(0,0), ~~ p_2=p_{2,0}= (0,4),~~ p_3=p_{3,0}= (1,3)
$$
are three framework vertices of a triangular subframework. Write
their translates as
\[
p_{\kappa, k}=T_kp_{\kappa,0}, \quad \mbox{ for } \kappa \in \{1,2,3\}, k\in \bZ,
\]
where $T_k$ is the isometry $T_k:(x,y)\to (x+4k,y)$ od $\bR^2$. The translation group $\T=\{T_k:k\in \bZ\}$
is also used to define a natural  periodic labelling of the framework edges:
\[
e_1=e_{1,0}=[p_1,p_2],~~ e_2=e_{2,0}=[p_2,p_3], ~~e_3=e_{3,0}=[p_1,p_3,],
\]
\[
 e_4=e_{4,0}
=[p_{1,0}, p_{1,1}], ~~e_5= e_{5,0}=[p_{3,0}, p_{2,1}]
\]
and
\[
e_{j,k}=T_ke_j,, \quad \mbox{ for } j \in \{1,2,3,4,5\}, k\in \bZ.
\]
Thus, the pair of finite sets
\[
F_v=\{p_1, p_2, p_3\},\quad F_e=\{e_1,\dots ,e_5\}
\]
have disjoint translates under $\T$ and the set of all such translates determines $\C$.

In an exactly similar way a translationally periodic bar-joint framework $\C$ in $\bR^d$ is determined by a triple $(F_v, F_e, \T)$ where we refer to the finite set pair
$\M=(F_v, F_e)$ as a \textit{motif}
for $\C$. Of particular interest for applications are the cases $d=2, 3$ in which $\T =\{T_k:k\in\bZ^d\}$ and
$\T$ has full rank. For $d=3$ "full rank" means that the so-called \textit{period vectors}
\[
a_1 =T_{(1,0,0)}(0),\quad a_2=T_{(0,1,0)}(0),\quad a_3= T_{(0,0,1)}(0)
\]
are linearly independent, in which case  the framework vertices,  if they are distinct, form a discrete set in $\bR^3$. We call such discrete bar-joint frameworks \textit{crystal frameworks}.

We now introduce a key dynamical ingredient, namely the notion of an \textit{infinitesimal flex}. This definition is the same as that for finite bar-joint frameworks being a specification of velocity vectors at the nodes which, to first order, do not change edge lengths.

\begin{defn}
Let $\C$ be crystal framework, with framework vertices $p_{\kappa,k}$ as above.
An infinitesimal flex of $\C$ is a set of vectors $u_{\kappa,k}$ (velocity vectors)
such that for each edge $e=[p_{\kappa,k},p_{\tau,l}]$
\[
\langle p_{\kappa,k}-p_{\tau,l},u_{\kappa,k}-u_{\tau,l}\rangle = 0.
\]
\end{defn}

The linear equations required for an infinitesimal flex $u=(u_{\kappa,k})$
translate to a single equation $R(\C)u=0$ where $R(\C)$ is the so called \textit{rigidity matrix} for the framework and where $u$ is a vector in the direct product vector space
$\H_v = \prod_{\kappa,k} \bR^d$, regarded as  a composite vector of  instantaneous velocities.
The rigidity matrix is sparse with rows labelled by edges and columns labelled by the
Euclidean coordinate labels
$(\kappa, k, \sigma)$ of the framework vertices, with $\sigma \in \{1, \dots ,d\}$;
the row  for framework edge
$e=[p_i,p_j]$ has the entry $(p_{\kappa,k}-p_{\tau,l})_\sigma$ for column $(\kappa,k,\sigma)$,
and has the negative of this entry for column $(\tau,l,\sigma)$.
Thus for $d=3$ row $e$  appears as
\[
[0\dots 0 ~~v_e~~ 0\dots 0 ~~-v_e~~ 0\dots 0]
\]
where the vector $v_e= p_{\kappa,k}-p_{\tau,l}$ (resp $-v_e$) is distributed in the columns for $(\kappa, k, \sigma)$
(resp. $(\tau, l, \sigma)$) with $\sigma \in \{x, y, z\}$.

From various viewpoints, such as phase-periodic velocity vectors on the one hand, or square-summable  velocity vectors on the other hand,  with the introduction
of complex scalars and functional representations of vector spaces, the rigidity matrix
$R(\C)$ leads to  a matrix-valued function $\Phi(z)$ with  $|F_e|$ rows and $d|F_e| $ columns.
The entries are scalar-valued functions on the $d$-torus of points $z= (z_1, \dots z_d)$
in $\bC^d$ with $|z_i|=1$.

We define this matrix function below in Definition \ref{d:Phi2} and one can check that  the strip framework of Figure 1 has associated matrix function
$$ \Phi(z)= \left[ \begin {array}{cccccc}
0 &-4&0&4&0&0\\
0 &0&-1&1&1&-1\\
-1 &-3&0&0&1&3\\
-4(1-\ol{z})&0&0&0&0&0\\
0&0&3\ol{z}&\ol{z}&-3&-1
\end {array} \right].
$$

\subsection{Examples from Seville}
The next three frameworks are based on some simple
two-dimensional tessellations that are suggested by tilings found in Seville cathedral
(Figures 2 and 3) and in the Alcazaar in Seville (Figure 4).
All three are in Maxwell counting equilibrium in the sense that the average
coordination number (average vertex degree) is $2$, matching the degree of freedom
of each vertex, while each subframework is not "overconstrained", in the sense that the number of edges
does not exceed twice the number of vertices.

A motif for the framework $\sevone$ is shown in Figure 3 together with the period vectors (dotted).
The motif edges consist of a square of edges together with two
vertical edges whose (equal) lengths fix the geometry up to a global scaling. (The "rigid units" of this framework are the single vertex subframeworks.)

Note that there is an evident (edge-length preserving) continuous flex, or deformation, $p(t)=(p_{\kappa,k}(t))$ of $\sevone$ which is associated with an expansion in the $x$ direction and a matching contraction in the $y$ direction. We remark that in the case of the geometry with all edge lengths equal this deformation passes through the framework composed of congruent rhombs which is reciprocal
(in the lattice sense \cite{cox}) to the well-known kagome framework, indicated in Section 5.1.
From the first instant of the deformation, so to speak, one obtains an infinitesimal flex
$u=p'(0)$ which (unlike infinitesimal translation flexes) is unbounded. Less evident are various
nontrivial bounded infinitesimal flexes, but we see below that there are plenty of these and that the RUM dimension is $1$.

\begin{center}
\begin{figure}[h]
\centering
\includegraphics[width=7cm]{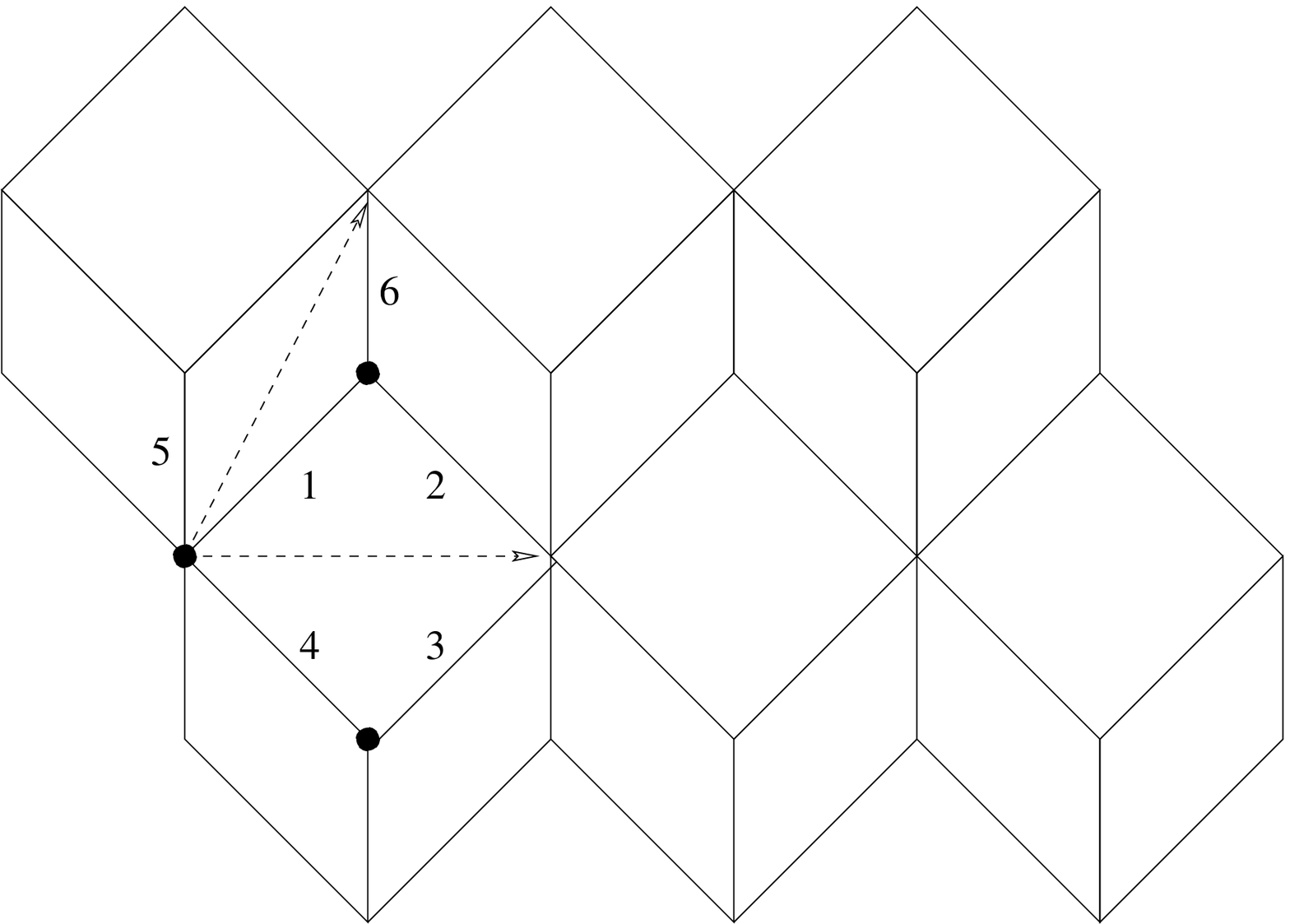}
\caption{The crystal framework $\sevone$.}
\end{figure}
\end{center}

The framework $\sevtwo$ in Figure 3 has triangular rigid units and a local
(finitely supported) infinitesimal flex is indicated, with four nonzero velocity vector components.
It is a general principle, as we note further below, that such a local phenomenon makes the framework
maximally flexible from a RUM point of view.

\begin{center}
\begin{figure}[h]
\centering
\includegraphics[width=7cm]{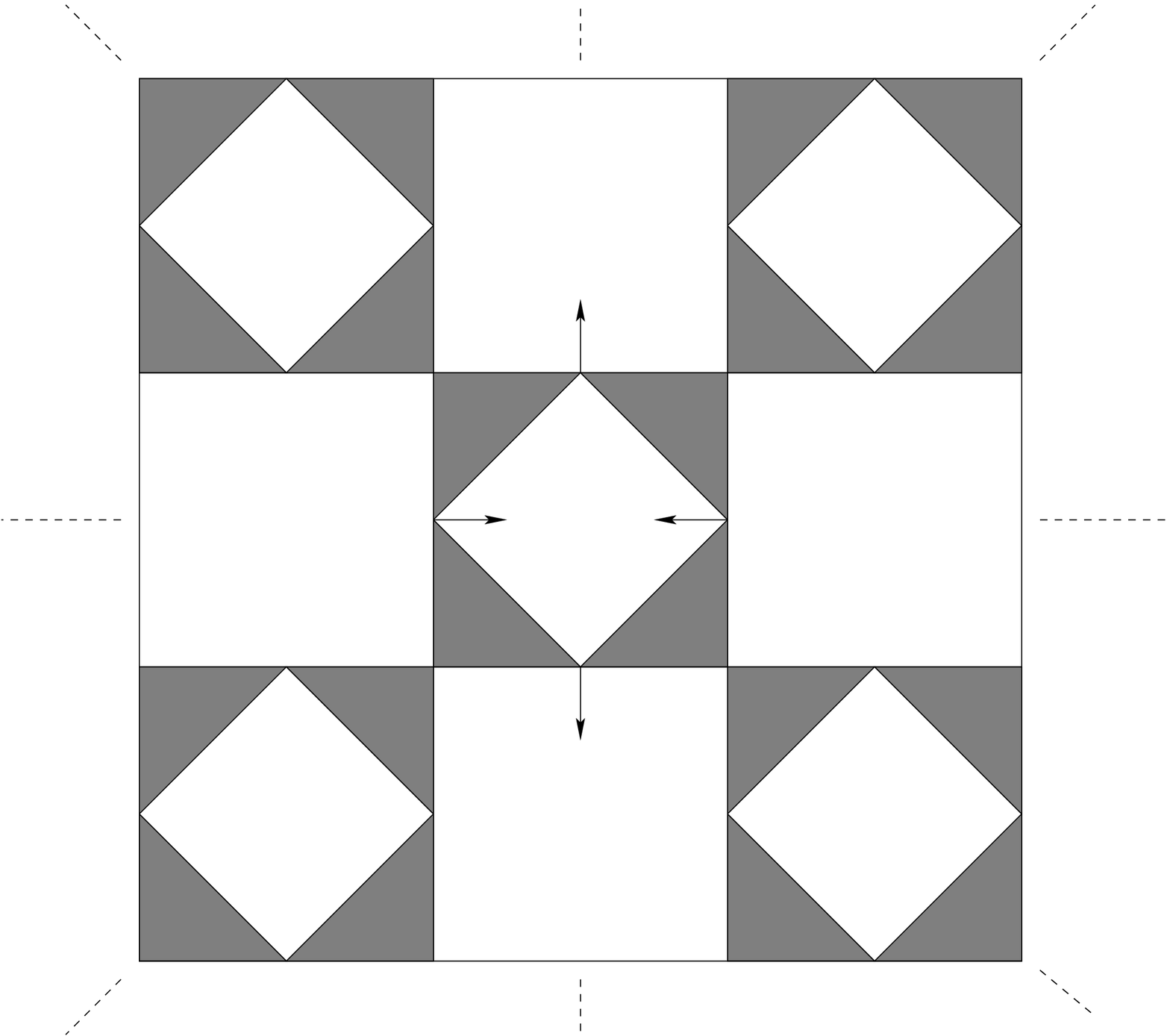}
\caption{$\sevtwo$, with a local infinitesimal flex.}
\end{figure}
\end{center}

The planar graph, or "topology" of this framework
is rather interesting, being a network of $4$-rings of triangles connected
"square-wise", so that the "holes" are $8$-cycles and $4$-cycles. One can make  similar constructions with equilateral triangles with this topology although now the local infinitesimal flex is lost.
We remark that periodic networks of pairwise corner-joined congruent equilateral triangles provide
the 2D variants of the tetrahedral nets associated with zeolites.
(That the hole cycles of
2D zeolites can be arbitrarily large
 follows from the substitution move which replaces each rigid unit triangle by a $3$-ring of smaller triangles with edge length halved.)

\begin{center}
\begin{figure}[h]
\centering
\includegraphics[width=7cm]{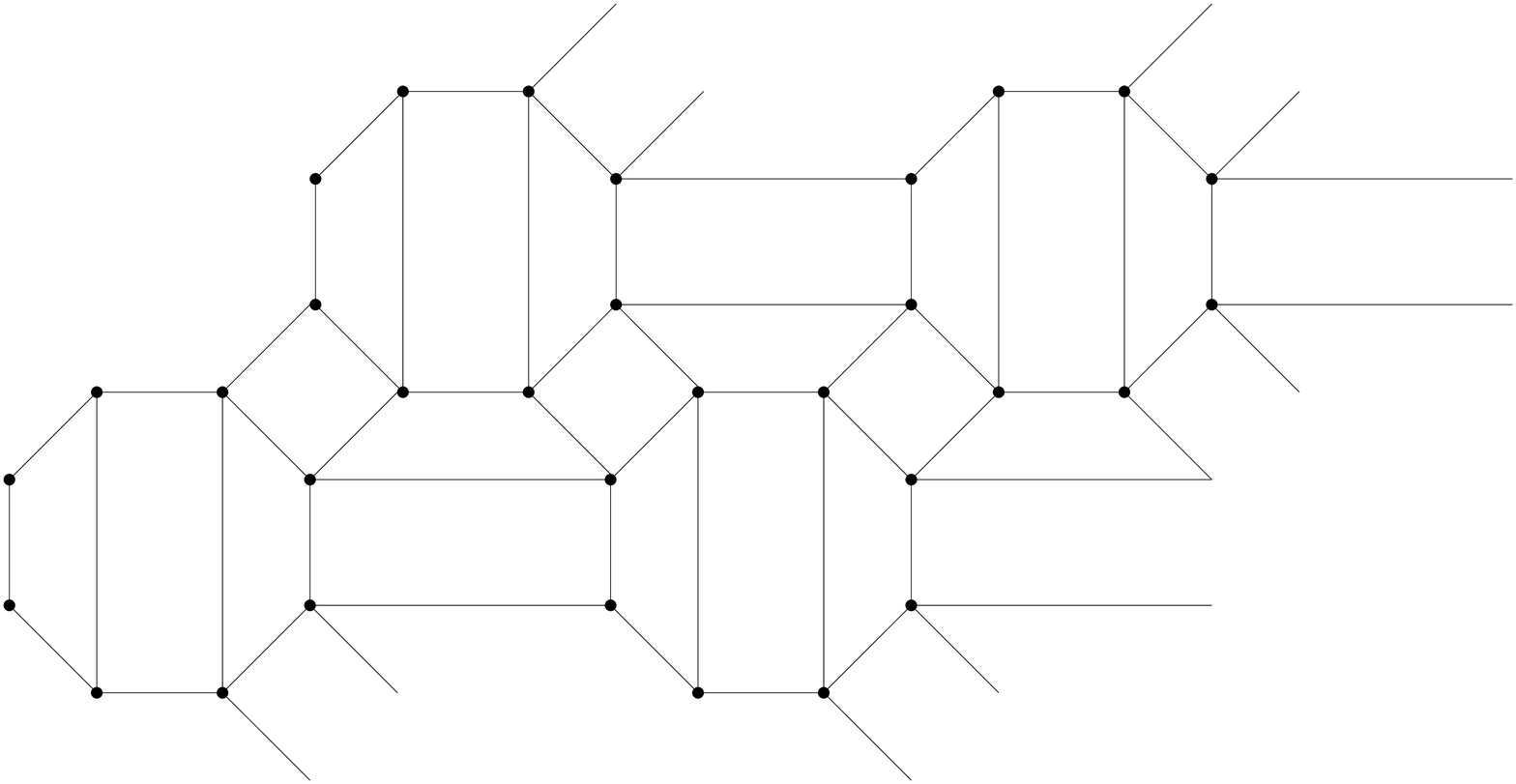}
\caption{Part of $\sevthree$, a homeomorph of $\C_{\bZ^2}$.}
\end{figure}
\end{center}

The framework $\sevthree$ is derived from a tiling in the Alcazaar in Seville.
A moment's thought reveals it to have the same topology (underlying graph) as that for the basic square grid framework $\C_{\bZ^2}$. The infinitesimal flexibility is less evident than is the case for $\C_{\bZ^2}$
but from the RUM viewpoint they turn out to be equally flexible with RUM dimension $1$.

\section{Bar-joint frameworks - a very brief overview.}

\subsection{Watt, Peaucellier, Cauchy, Euler, Kempe, Maxwell, Laman.} Informally,
a "linkage" is a bar-joint framework with one degree of essential flexibility.
In 1784 James Watt designed a bar-joint linkage which transformed circular motion
into approximate linear motion.
This was rather important for steam engine transmission.
The mechanism was approximate
and was superseded by  Peaucellier's exact linear motion linkage eighty years later. In
1876 Kempe \cite{kem} solved the general inverse problem by showing that any finite algebraic curve can be simulated by a linkage.
The rigidity of geometric frameworks was also of interest to Euler and to Cauchy
who in particular were concerned with the rigidity of polyhedra with hinged faces; a beautiful classical result is the infinitesimal rigidity of all convex triangle-faced polyhedra.

James Clerk Maxwell initiated combinatorial aspects with the observation that a graph $G=(V,E)$ with a minimally rigid generic framework realisation in the plane must satisfy the simple counting rule $2|V|-|E| =3$ together with the inequalities  $2|V'|-|E'| \geq 3$ for all subgraphs. The number $3$ represents the number of independent global infinitesimal motions for the plane.
In 1970 Laman \cite{lam} obtained the fundamental result that Maxwell's
conditions are sufficient for generic rigidity and this result also anticipated the advent
of matroid theory in rigidity theory. While corresponding counting rules
are necessary in three dimensions they fail to be sufficient and no necessary and sufficient combinatorial conditions for generic rigidity are known ! For further information see  \cite{asi-rot}, \cite{asi-rot-2}, \cite{gra-ser-ser}.

\subsection{Some Recent work}
Laman's theorem concerns \textit{generic} frameworks with a particular graph.
One can expect that special frameworks with global symmetries may have more flexibility and this is a topic of current interest. See for example, Connelley et al \cite{con-et-al}, Owen and Power \cite{owe-pow-fsr} and Shulze
\cite{sch-1}, \cite{sch-2}.
Understanding constraint systems of geometric objects with symmetry present is also of significance
for algorithms for CAD software \cite{owe-pow-fsr}.

Laman's theorem is also concerned with \textit{finite} frameworks. A natural generalisation,
also of significance for applications, are periodic frameworks in the sense of the
crystal frameworks above. See also Whiteley \cite{whi-union}, Borcea and Streinu \cite{bor-str},
Malestein and Theran \cite{mal-the}, and Ross, Schulze and Whiteley \cite{ros-sch-whi}.

The theory of general infinite bar-joint frameworks, from the point of view of rigidity, is a novel topic
and perhaps a rather curious one. We point out in Owen and Power  \cite{owe-pow-kempe},\cite{owe-pow-hon} that it is possible to generalise Kempe's theorem to the effect that any \textit{continuous} curve
(ie. continuous image of $[0,1]$ in $\bR^2$) may
be simulated by an infinite linkage. (In \cite{owe-pow-kempe} this is achieved with three vertices of infinite degree but in fact infinite degree vertices are not necessary.)

In  material science the microporous flexing materials known as zeolites are on the one hand important for industrial applications, as filters, and on the other hand present diverse tetrahedral rigid unit frameworks. The degree of continuous flexibility of such idealised zeolites is investigated in Kapko et al \cite{kap-daw-tre-tho}.

\section{The RUM spectrum and RUM dimension of $\C$}
Let $\C=(F_v, F_e, \T)$ be a crystal framework in $\bR^d$ and
let $\K_v$ be the vector space  $\prod_{\kappa, k}\bC^{d|F_v|}$ consisting of  infinitesimal velocity vectors. Let $\bT^d$ be the $d$-torus of points $\omega=(\omega_1, \dots ,\omega_d)$ and for $k \in \bZ^d$ write $\omega^k$ for the unimodular complex numbers $\omega_1^{k_1}\dots\omega_d^{k_d}$.

\begin{defn} (\cite{owe-pow-crystal}, \cite{pow-matrix}.)
(a) A velocity vector $\tilde{u}$ in $\K_v$ is  periodic-modulo-phase
for the (multi-)phase factor  $\omega \in \bT^d$  if there exists a vector $u=(u_\kappa)$ in $\bC^{|F_v|}$  such that
\[
\tilde{u}_{\kappa,k} = \omega^ku_\kappa, \quad \kappa\in F_v, k\in \bZ^d.
\]
Also $\K_v^\omega$ denotes the associated vector subspace of such vectors.

(b) A periodic-modulo-phase infinitesimal flex (or wave flex) is a vector $\tilde{u}$ in $\K_v^\omega$ which is an infinitesimal flex for $\C$.

(c) The rigid unit mode spectrum, or  RUM spectrum, of $\C$ (with specified translation group $\T$)  is the set $\Omega(\C)$ of phases for which there exists a nonzero periodic-modulo-phase infinitesimal flex.
\end{defn}

To each multiphase $\omega$ there exists a unique wave vector
${\bf k}=({\bf k}_1,\dots ,{\bf k}_d)$ such that
$\omega_j =
e^{2\pi i~{\bf k_j}}, 1 \leq j \leq d$.
Ignoring bond constraints for the moment, recall that the framework points  of $\C$ undergo harmonic motion with wave vector ${\bf k}$ when the vertex positions at time $t$ satisfy equations of the form
\[
p_{\kappa, k}(t) =  p_{\kappa, k} + \exp(2\pi i~{\bf k}\cdot k)\exp(i\alpha t)v_\kappa,
\]
where $\alpha/2\pi$ is the frequency of the oscillation and where ${\bf k}\cdot k $ is the inner product ${\bf k}_1\cdot k_1+ \dots +{\bf k}_d\cdot k_d$. Such pure motions appear as basic solutions in lattice dynamics, under harmonic approximation, with general solutions obtained by linear superposition.
(See Dove \cite{dov-book} for example.)

The following theorem from \cite{pow-matrix} provides an explanation for the
connection between low energy oscillation modes alluded to in Section 2 and infinitesimal wave flexes.

\begin{thm}\label{p:LongWavengthLimit}
Let $\C$ be a crystal framework, with specified periodicity,  and let ${\bf k}$ be a wave vector with point $\omega \in\bT^d$.
Then the following assertions are equivalent.

(i) $(\omega^ku_\kappa)_{\kappa,k}$ is a nonzero periodic-modulo-phase infinitesimal (complex) flex for $\C$.

(ii)
For the vertex wave motion
\[
p_{\kappa, k}(t) =  p_{\kappa, k} + u_\kappa\exp(2\pi i~{\bf k}\cdot k)\exp(i\alpha t),
\]
and a given time interval, $t\in [0,T]$,
the bond length changes
\[
|p_{\kappa, k}(t)- p_{\tau, l}(t)|-|p_{\kappa, k}(0)- p_{\tau, l}(0)|,
\]
for the edges $e$
tend to zero uniformly, in both $t$ and $e$, as the wavelength $2\pi /\alpha$ tends to infinity.
\end{thm}

Next we define the matrix function $\Phi_\C$ of a crystal framework $\C$ with given period vectors.
For the multi-index $k=(k_1,\dots ,k_d)$  write $z^k$ for the usual monomial function
on $\bT^d$.

\begin{defn}\label{d:Phi2}
Let $\C$ be a crystal framework in $\bR^d$, with  motif sets
\[
F_v=\{p_{\kappa,0}:1\leq\kappa\leq |F_v|\}, \quad F_e=\{e_i:1\leq i \leq |F_e|\},
\]
and for each edge $e=[p_{\kappa,k}, p_{\tau,l}]$ in $F_e$
let $v_e$ be the edge vector $p_{\kappa,k}-p_{\tau,l}$.
The matrix-valued function $\Phi_\C(z)$ has rows labelled by the edges $e\in F_e$ and
has $d|F_v|$ columns labelled by $\kappa$ and the coordinate index $\sigma \in\{1,\dots ,d\}$.
If $\kappa \neq \tau$ for edge $e$ in $F_e$ then
\[
(\Phi_\C(z))_{e,(\kappa,\sigma)} = (v_e)_\sigma\bar{z}^k,
\]
\[
(\Phi_\C(z))_{e,(\tau,\sigma)}  = -(v_e)_\sigma\bar{z}^l,
\]
while for each reflexive edge $[p_{\kappa,0}, p_{\kappa,\delta(e)}]$
\[
(\Phi_\C(z))_{e,(\kappa,\sigma)}  = (v_e)_\sigma(1-\bar{z}^{\delta(e)}),
\]
with the remaining entries in each row equal to zero.
\end{defn}
The next theorem gives one connection between
$\Phi_\C(z)$ and the infinitesimal flex properties of $\C$.
Here we view the rigidity matrix $R(\C)$ as a linear
transformation from the product vector space $\K_v = \prod_{\kappa, k}\bC^{d}$
to the edge vector space $\K_e = \prod_{\rm edges}\bC = \prod_{e,k}\bC$.

\begin{thm}\label{t:phaserigiditymatrix}
The restriction of the rigidity matrix $R(\C)$ to the finite dimensional vector space
$\K_v^\omega$ has  representing matrix $\Phi_\C(\ol{\omega})$ with respect to natural
vector space bases.
\end{thm}

\begin{proof} Let $\tilde{u}$ be a velocity vector in $\K_v^\omega$ determined by $u \in \bC^{d|F_v|}$ as above.
Let $e$ in $F_e$ be an edge of the form $[p_{\kappa,k},p_{\tau,k+\delta}]$.
Let $\langle \cdot , \cdot \rangle$ denote the  bilinear form on $\bC^d$. The $(e,k)^{th}$ entry of $R(\C)\tilde{u}$ can be written  as
\[
(R(\C)\tilde{u})_{e,k}= \langle v_e, \tilde{u}_{\kappa,k}  \rangle - \langle v_e, \tilde{u}_{\tau,k+\delta}  \rangle
\]
\[
= \langle v_e, \omega^ku_\kappa\rangle  - \langle v_e,  \omega^{k+\delta}u_\tau  \rangle
\]
\[
={\omega}^k(\langle v_e, u_\kappa\rangle + \langle -{\omega}^{\delta}v_e,u_\tau \rangle),
\]
and one can check that this agrees with $
{\omega}^k(\Phi_\C(\ol{\omega})u)_e$, both in the case $\kappa \neq \tau$ and when $\kappa=\tau$.
\end{proof}

We can now  identify  the \textit{RUM spectrum}
of $\C$ as the algebraic variety in $\bT^d$ given by
\[
\Omega(\C) = \{z:\rank\Phi_\C(z) < |F_e|\}.
\]
This set does depend on the choice of translation group. One could  define the \textit{primitive RUM spectrum} to correspond to the translation group for a primitive unit cell, and this is then well-defined, up to coordinate permutations.  If one doubles the period vectors, and hence the unit cell,  then the new
spectrum is obtained simply as the range of the old spectrum under the doubling map $(w_1, w_2, w_3) \to
(w_1^2, w_2^2, w_3^2)$.

While we have given the multiphase form of the RUM spectrum the convention in material science is to indicate such a spectrum (in three dimensions) as the set of (reduced) wave vectors ${\bf k}$ in the unit cube $[0,1)^3$.
For calculations of RUM spectrum by different methods see Dove et al \cite{dov-exotic}
(simulation calculations), Wegner \cite{weg} (computer algebra calculations), Owen and Power \cite{owe-pow-crystal} and Power \cite{pow-matrix} (direct
calculations).


The algebraic variety perspective of RUMs appears to be new and opens the way for new methods
and terminology for understanding the curious curved surfaces in \cite{dov-exotic}. For example, it is natural to define the \textit{RUM dimension} of $\C$ to be
the topological dimension of $\Omega (\C)$ as a real algebraic variety. By the comments above on unit cell doubling it follows that this quantity is independent of the translation group.

Tetrahedral net frameworks in 3 dimensions, with pairwise vertex connection, satisfy {Maxwell counting equilibrium} and in the periodic case, with no penetrating tetrahedra,  are sometimes referred to as hypothetical zeolites \cite{fos-tre}.
(In material crystalline zeolites the rigidly bonded ~ SiO$_4$  units make up such a bond-node framework.)
Even in this case all possibilities occur for the RUM dimension, namely $0, 1, 2, 3$, and this depends, roughly speaking, on the degree of symmetry of the framework. In particular as we note below the framework for the cubic form of sodalite indicated below has full RUM spectrum, corresponding to dimension 3.
(This so-called order $N$ property of sodalite was first observed experimentally. See \cite{ham-dov-zeo1997}.)

For crystal frameworks in Maxwell counting equilibrium the matrix function is square and the RUM spectrum is revealed, in theory at least, as the intersection of the zero set of  the multi-variable polynomial
$\det \Phi_\C(z)$ with the $d$-torus $\bT^d$. In fact, after fixing a monomial order on the
$d$ indeterminates $z_1, \dots ,z_d$ one may  formally define the \textit{crystal polynomial} $p_\C(z)$, associated with $\C$. (See \cite{pow-matrix}.) This is given
as the product $z^\gamma\det(\Phi_\C(z))$ where the monomial exponent $\gamma$ is chosen so that

(i) $p_\C(z)$ is a  linear combination of nonnegative power monomials,
\[
p_\C(z)= \sum_{\alpha \in \bZ^d_+} a_\alpha z^\alpha,
\]

(ii)  $p_\C(z)$  has minimum total degree, and

(iii)  $p_\C(z)$ has leading monomial with coefficient $1$.
\medskip

The RUM spectrum certainly has symmetry reflecting the crystallographic
group symmetries of the crystal framework.
Even so the point group may be trivial and the following abstract inverse problem (a Kempe theorem for RUMS ?) may well have an affirmative answer.
\medskip

\noindent \textit{Problem.} Let $q(z,w)$ be a polynomial with real coefficients with $q(1,1)=0$.
Is there a crystal framework with crystal polynomial $p(z,w)$ whose zero set on the $2$-torus is the same as that for $q(z,w)$ ?
\medskip

\subsection{Examples}
The tiling-derived framework of Figure 3 has vanishing crystal polynomial. Indeed, this can be predicted from the existence of a local infinitesimal flex. Such a flex allows the construction of infinitesimal phase-periodic flexes for all phases and so the zero set
of the polynomial includes the entire $2$-torus.

The tiling-derived framework of Figure 2 is also in Maxwell counting equilibrium, its symbol
function is $6 \times 6$ and one can show by direct calculation
 that if we write the indeterminates in this case as $z, w$, then the crystal polynomial
 is
 \[
 (w-1)(z^2+z(1+w)+w)
 \]

The  \textit{kagome framework}  is the framework $\Ckag$ formed by pairwise corner
connected equilateral triangles in regular hexagonal arrangement. Its symbol function
is a $6 \times 6$ matrix and, if we write the indeterminates in this case as $z, w$, then the
crystal polynomial is
\[
(z-1)(w-1)(z-w).
\]

There is a natural $3D$ variant of the kagome lattice known as the \textit{kagome net}.
The corresponding crystal framework $\Cknet$ has period vectors formed by three edges
of a parallelapiped at pairwise angles of $\pi/3$, and each parallelapiped contains two tetrahedral rigid units such that the three planar slices of $\Cknet$ for each pair of period vectors, is a copy of $\Ckag$.
The crystal polynomial takes the form
\[
p(z,w, u) = (z - 1)(w - 1)(u - 1)(z - w)(w - u)(z - u).
\]

 The factorisations in these examples makes evident the nature
of the RUM spectrum as a union of lines and a union of surfaces, respectively.
In fact the individual factors can be predicted in terms of the identification of infinitesimal flexes that are supported within a linear band (for the $2D$ case) or a linear tube (in the $3D$ case). See \cite{pow-matrix}.
For considerably more complicated polynomials with nonlinear "exotic" spectrum see Wegner \cite{weg} and
Power \cite{pow-matrix}.

Figure 4 shows a $4$-ring of tetrahedra, three copies of which, placed on three adjacent sides of an imaginary cube, provide the edges and vertices for a motif $(F_v, F_e)$ for the framework $\Csod$ for the cubic form of sodalite. (From a mathematical perspective, this structure is arguably the most elegant of
the naturally  occurring zeolite framework types \cite{bae-mcc}.)

A full set of eight $4$-rings forms a so-called sodalite cage.
With the $24$ outer vertices of this cage fixed there is nevertheless an infinitesimal flex of the structure and so a local infinitesimal flex of $\Csod$. This is in analogy with the framework $\sevtwo$. It follows that the determinant of the symbol function (a $72 \times 72$ sparse function matrix) vanishes identically and that  the sodalite framework $\Csod$ has RUM dimension $3$.

\begin{center}
\begin{figure}[h]
\centering
\includegraphics[width=6
cm]{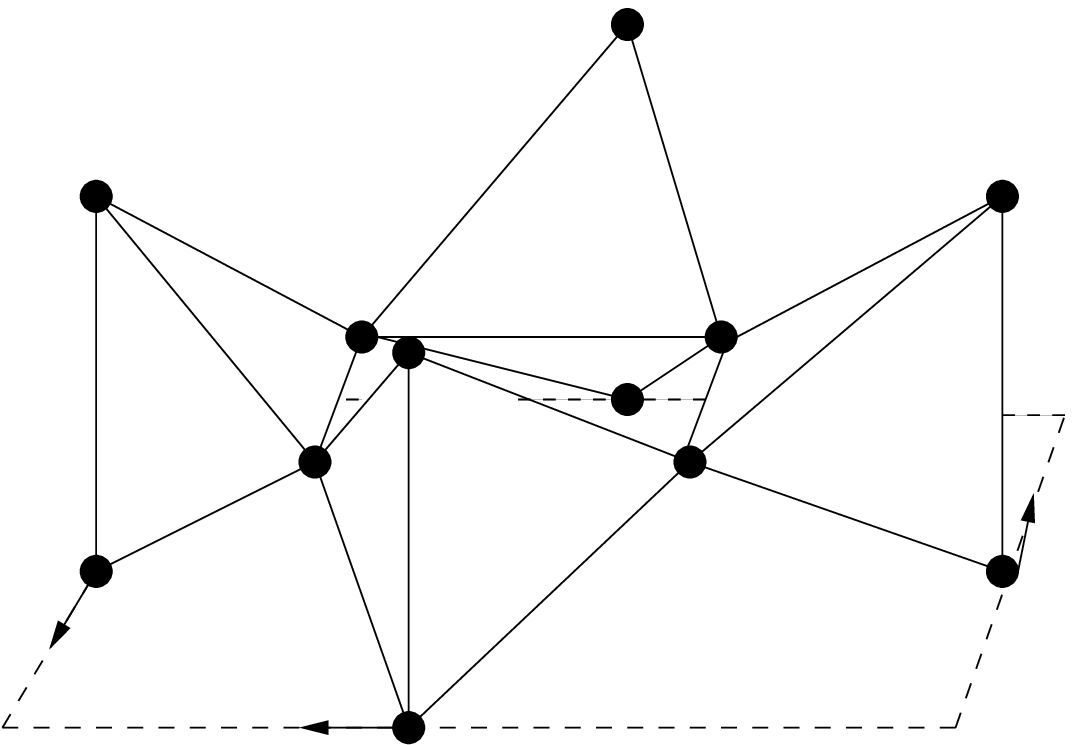}
\caption{A $4$-ring building unit from $\Csod$.}
\end{figure}
\end{center}

\section{Flexes with decay and Toeplitz operators.}
Let $\C$ be a crystal framework with an implicit choice of translational
periodicity.
Write $\K_v^2$ and $ \K_e^2$ for the Hilbert spaces of square-summable sequences in
$\K_v$ and  $\K_e$. Then $R(\C)$ determines a bounded Hilbert space operator
from $\K_v^2$ to  $\K_e^2$.
The natural Hilbert space basis for $\K_v^2$, associated with the given periodicity, may be denoted
$\{\xi_{\kappa, \sigma, k}\}$, where $\sigma$ ranges from $1$ to $d$. Similarly, the basis for  $\K_e^2$ is $\{\eta_{e,k}\}$, with $e\in F_e, k \in \bZ^d$.

Regarding such
square-summable sequences as the Fourier series of square integrable vector-valued functions one obtains unitary operators
$$U_v: \K_v^2 \to L^2(\bT^d)\otimes \bC^{d|F_v|}$$
and
$U_e: \K_e^2\to L^2(\bT^d)\otimes \bC^{|F_e|}$. In the next theorem, from \cite{owe-pow-crystal}, the unitary equivalence
referred to is two-sided and the equivalence in question is the operator identity
$U_e^*R(\C)U_v=M_{\Phi_\C}$. That $U_e^*R(\C)U_v$ has the form of a multiplication operator $M_\Psi$ follows from standard operator theory, since
this operator intertwines the canonical shift operators.
Borrowing standard operator terminology we refer to $\Phi_C$ as the \textit{symbol function} of $\C$.

\medskip

\begin{thm}\label{t:rigidityoperator}
The infinite rigidity matrix $R(\C)$ of the crystal framework $\C$ in $\bR^d$ determines a Hilbert space operator which is unitarily equivalent  to the multiplication operator
\[
M_{\Phi_\C}: L^2(\bT^d)\otimes \bC^{d|F_v|} \to L^2(\bT^d)\otimes \bC^{|F_e|}.
\]
where  $\Phi_C$ is the matrix function for $\C$.
\end{thm}

The following corollary follows from elementary operator theory.

\begin{cor} A crystal framework  has a square-summable infinitesimal flex if and only if its
symbol function has reduced column rank on a set of positive measure.
\end{cor}

It is natural to ask whether crystal frameworks possess infinitesimal flexes which decay to zero at infinity.
Note that if an infinitesimal linear subframework has an infinitesimal flex with such decay then the flex velocities must be orthogonal to the direction of this subframework. (For otherwise there must be identical nozero velocity components in that direction on all the subframework points.)
For this reason it follows that the kagome framework and similar "linear" frameworks have no asymptotically vanishing flexes and in particular, no square-summable flexes.
In fact one can exploit the matrix function formalism to obtain the following much more general fact.

\begin{thm}\label{t:internalflexthm}
\cite{owe-pow-crystal}
The following are equivalent for a crystal framework $\C$ with Maxwell counting equilibrium.

(i) $\C$ has a nonzero internal ("finitely-supported", "local") infinitesimal flex

(ii) $\C$ has a nonzero summable infinitesimal flex.

(iii) $\C$ has a nonzero square-summable infinitesimal flex.
\end{thm}

\subsection{Semi-infinite and bi-crystal frameworks}
We may define a \textit{semi-infinite crystal framework} $\D$ as a subframework of a crystal framework $\C$ with an exposed face. More formally $\D$ is supported by the framework vertices that lie in a half-space which is invariant
under a subsemigroup of an underlying translation group for $\C$.
In the case of planar frameworks this may be specified in the form of a triple $(F_v, F_e, \T_+)$ where
$(F_v, F_e)$ is an appropriate motif and $\T_+$ is a subsemigroup of $\T$ isomorphic to one of
$\bZ_+\times \bZ$,
$\bZ_-\times \bZ$,
$\bZ\times \bZ_+$,
$\bZ\times \bZ_-$.
It is not hard to verify that the rigidity operators of semi-infinite crystal frameworks can be identified with various Toeplitz operators derived from $M_{\Phi_\C}$ by compression to Hardy space Hilbert spaces, such as $H^2(\bT)\otimes L^2(\bT)$ in the case of $\bZ_+\times \bZ$.

We remark that semi-infinite frameworks have rigidity matrices that feature as block submatrices of the rigidity matrices
of \textit{bi-crystal frameworks}. By this we mean (for example) a framework obtained in three dimensions
by identifying two semi-infinite frameworks at their common surface of vertices, when this is possible.
It seems that Toeplitz operators  could provide a useful formalism for their analysis.\footnote{
Added October 2011. It appears, from \cite{sut-bal} for example, that there is extensive interest in material bicrystals.}

For semi-infinite crystal frameworks the equivalence of the previous theorem no longer holds and we illustrate this with a simple variant of the strip framework in Figure 1 whose matricial symbol function $\Phi(z)$ is as given in Section 2. Consider
the submatrix function $\Phi_0(z)$ obtained on removing the first two columns, corresponding to the "supporting" framework vertices, and removing the row corresponding to the "base" edges. The degeneracies of this matrix function correspond to the phases of periodic-modulo-phase infinitesimal flexes
which do not deflect the "supporting" vertices.
We have
$$ \Phi_0(z)= \left[ \begin {array}{ccccc}
0&4&0&0\\
-1&1&1&-1\\
0&0&1&3\\
3\ol{z}&\ol{z}&-3&-1
\end {array} \right].
$$
The determinant, $16(2-3\ol{z})$, does not vanish on $\bT$ and so
the "base-rooted" framework is infinitesimally rigid from the point of view of phase-periodic
infinitesimal flexes. Since the determinant certainly does not vanish
on a set of positive measure on $\bT$, there are no square-summable infinitesimal flexes
which fix the baseline vertices of the framework.
On the other hand there is an unbounded infinitesimal flex, corresponding to a two-way
infinite geometric series for
$z=2/3$ and this unbounded flex  reflects the concatenated lever structure of the framework.

 This analysis also applies to the strip framework  in Figure 6. One can readily check that
  up to scalar multiplication there is a unique proper
(unbounded) infinitesimal flex. (Note incidentally, that this flex does not
extend to a continuous flex. Put another way, each finite strip subframework
here is continuously flexible but the complete two-way infinite
framework is not. This kind of phenomenon is referred to as \textit{vanishing flexibility}
in \cite{owe-pow-crystal} and can occur in more subtle ways.)
\bigskip

\begin{center}
\begin{figure}[h]
\centering
\includegraphics[width=7cm]{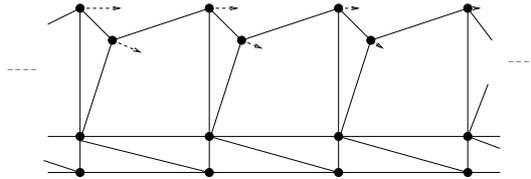}
\caption{An unbounded infinitesimal flex.}
\end{figure}
\end{center}
\bigskip

There are two natural semi-infinite frameworks $\C_+, \C_-$ associated with the strip framework
of Figure 6, namely
the right strip and the left strip. Each has a infinite triangulated rigid base framework which supports linked triangles. The former has a square-summable flex while the latter does not. This fact is evident by elementary direct analysis and in fact can be viewed as a reflection of the contrasting nature of analytic and coanalytic Toeplitz operators. We expect such operator theory to play a useful role in the analysis of more complex examples with larger unit cells, and in the analysis of surface phonons and surface phenomena
in semi-infinite structures.

\bibliographystyle{abbrv}
\def\lfhook#1{\setbox0=\hbox{#1}{\ooalign{\hidewidth
  \lower1.5ex\hbox{'}\hidewidth\crcr\unhbox0}}}

\end{document}